\documentclass[12pt]{article}
\usepackage{amsfonts}
\usepackage{amsmath}
\usepackage{amssymb}
\usepackage{amsthm}
\usepackage{enumerate}
\makeatletter
\def\imod#1{\allowbreak\mkern10mu({\operator@font mod}\,\,#1)}
\makeatother

\newtheorem{theorem}{Theorem}[section]
\newtheorem{lemma}{Lemma}[section]
\newtheorem{corollary}{Corollary}[section]

\newtheorem{conjecture}{Conjecture}[section]

\theoremstyle{definition}
\newtheorem{definition}{Definition}[section]
\newtheorem{remark}{Remark}[section]

\begin{document}
\begin{center}
\vskip 1cm{\LARGE\bf On Sparsely Schemmel Totient Numbers 
\vskip 1cm
\large
Colin Defant\footnote{This work was supported by National Science Foundation grant no. 1262930.}\\
Department of Mathematics\\
University of Florida\\
United States\\
cdefant@ufl.edu}
\end{center}
\vskip .2 in

\begin{abstract} 
For each positive integer $r$, let $S_r$ denote the $r^{th}$ Schemmel totient function, a multiplicative arithmetic function defined by 
\[S_r(p^{\alpha})=\begin{cases} 0, & \mbox{if } p\leq r; \\ p^{\alpha-1}(p-r), & \mbox{if } p>r \end{cases}\] 
for all primes $p$ and positive integers $\alpha$. The function $S_1$ is simply Euler's totient function $\phi$. Masser and Shiu have established several fascinating results concerning sparsely totient numbers, positive integers $n$ satisfying $\phi(n)<\phi(m)$ for all integers $m>n$. We define a sparsely Schemmel totient number of order $r$ to be a positive integer $n$ such that $S_r(n)>0$ and $S_r(n)<S_r(m)$ for all $m>n$ with $S_r(m)>0$. We then generalize some of the results of Masser and Shiu. 
\end{abstract} 

\section{Introduction} 
Throughout this paper, we will let $\mathbb{N}$ and $\mathbb{P}$ denote the set of positive integers and the set of prime numbers, respectively. For any prime $p$ and positive integer $n$, we will let $\omega(n)$ denote the number of distinct prime factors of $n$, and we will let $\upsilon_p(n)$ denote the exponent of $p$ in the prime factorization of $n$. Furthermore, we will let $n\#$ denote the product of all the prime numbers less than or equal to $n$ (with the convention $1\#=1$), and we will let $p_i$ denote the $i^{th}$ prime number. 
\par 
The Euler totient function $\phi(n)$ counts the number of positive integers less than or equal to $n$ that are relatively prime to $n$. In 1869, V. Schemmel introduced a class of functions $S_r$, now known as Schemmel totient functions, that generalize Euler's totient function. $S_r(n)$ counts the number of positive integers $k\leq n$ such that $\gcd(k+j,n)=1$ for all $j\in\{0,1,\ldots,r-1\}$. Clearly, $S_1=\phi$. It has been shown \cite{schemmel69} that $S_r$ is a multiplicative function that satisfies 
\[S_r(p^{\alpha})=\begin{cases} 0, & \mbox{if } p\leq r \\ p^{\alpha-1}(p-r), & \mbox{if } p>r \end{cases}\] 
for all primes $p$ and positive integers $\alpha$. For any positive integer $r$, we will let $B_r$ denote the set of positive integers whose smallest prime factor is greater than $r$, and we will convene to let $1\in B_r$. Equivalently, \[B_r=\{n\in\mathbb{N}\colon S_r(n)>0\}.\]  
\par 
Masser and Shiu have studied the set $F$ of positive integers $n$ that satisfy $\phi(n)<\phi(m)$ for all $m>n$ \cite{Masser86}. These integers are known as sparsely totient numbers, and they motivate the following definition. 
\begin{definition} \label{Def1.1} 
Let $r$  be a positive integer. A positive integer $n$ is a \textit{sparsely
Schemmel totient number of order $r$} if $n\in B_r$ and $S_r(n)<S_r(m)$ for all $m\in B_r$ with $m>n$. We will let $F_r$ be the set of all sparsely Schemmel totient numbers of order $r$. 
\end{definition} 
\begin{remark} \label{Rem2.1} 
Lee-Wah Yip has shown that if $r$ is a positive integer, then there exists a positive constant $c_1(r)$ such that $\displaystyle{S_r(n)\geq\frac{c_1(r)n}{(\log\log 3n)^r}}$ for all $n\in B_r$ \cite{Yip89}. Therefore, each set $F_r$ is infinite.  
\end{remark}  
\par 
The aim of this paper is to modify some of the proofs that Masser and Shiu used to establish results concerning sparsely totient numbers in order to illustrate how those results generalize to results concerning sparsely Schemmel totient numbers. 
\section{A Fundamental Construction} 
The fundamental result in Masser and Shiu's paper, upon which all subsequent theorems rely, is a construction of a certain subset of $F$, so we will give a similar construction of subsets of the sets $F_r$. 
\begin{lemma} \label{Lem2.1}
Fix some positive integer $r$, and suppose $x_1,x_2,\ldots,x_s,y_1,$ $y_2,\ldots,y_s,X,Y$ are real numbers such that $r<x_i\leq y_i$ for all $i\in\{1,2,\ldots,s\}$. If $Y\geq\max(x_1,x_2,\ldots,x_s)$ and $\displaystyle{X\prod_{i=1}^sx_i<Y\prod_{i=1}^sy_i}$, then \[(X-r)\prod_{i=1}^s(x_i-r)<(Y-r)\prod_{i=1}^s(y_i-r).\]   
\end{lemma} 
\begin{proof} 
The proof is by induction on $s$, so we will assume that $s\geq 2$ and that the lemma is true if we replace $s$ with $s-1$. Note that \[\prod_{i=1}^s(x_i-r)\leq(Y-r)\prod_{i=1}^{s-1}(y_i-r),\] so the proof is simple if $X<y_s$. Therefore, we will assume that $X\geq y_s$. If we write the inequality $\displaystyle{X\prod_{i=1}^sx_i<Y\prod_{i=1}^sy_i}$ as 
$\displaystyle{\frac{Xx_s}{y_s}\prod_{i=1}^{s-1}x_i<Y\prod_{i=1}^{s-1}y_i}$, then the induction hypothesis tells us that 
\begin{equation} \label{Eq2.1} 
\left(\frac{Xx_s}{y_s}-r\right)\prod_{i=1}^{s-1}(x_i-r)<(Y-r)\prod_{i=1}^{s-1}(y_i-r).
\end{equation}  
Multiplying each side of \eqref{Eq2.1} by $y_s-r$, we see that it suffices to show, in order to complete the induction step, that 
\begin{equation} \label{Eq2.2} 
\left(\frac{Xx_s}{y_s}-r\right)(y_s-r)\geq(X-r)(x_s-r). 
\end{equation}
We may rewrite \eqref{Eq2.2} as $\displaystyle{r\left(\frac{Xx_s}{y_s}+y_s\right)\leq r(X+x_s)}$, or, equivalently, \\ $\displaystyle{y_s-x_s\leq X\left(1-\frac{x_s}{y_s}\right)}$. This inequality holds because $X\geq y_s$, so we have completed the induction step of the proof. 
\par 
For the case $s=1$, we note again that the proof is trivial if $X<y_1$, so we will assume that $X\geq y_1$. This implies that $\displaystyle{y_1-x_1\leq X\left(1-\frac{x_1}{y_1}\right)}$, which we may rewrite as $\displaystyle{y_1+\frac{Xx_1}{y_1}\leq X+x_1}$. Multiplying this last inequality by $-r$ and adding $Xx_1+r^2$ to each side, we get \[Xx_1-r\left(y_1+\frac{Xx_1}{y_1}\right)+r^2\geq Xx_1-r(X+x_1)+r^2,\] so $\displaystyle{(y_1-r)\left(\frac{Xx_1}{y_1}-r\right)\geq(x_1-r)(X-r)}$. As $x_1X<y_1Y$ by hypothesis, we find that $(y_1-r)(Y-r)>(x_1-r)(X-r)$.   
\end{proof} 
In what follows, we will let $b(1)=0$, and, for $r\geq 2$, we will let $b(r)$ denote the largest integer such that $p_{b(r)}\leq r$.
\begin{theorem} \label{Thm2.1} 
Let $r$ be a positive integer, and let $\ell$ and $k$ be nonnegative integers such that 
$k\geq b(r)+2$. Suppose $d$ is an element of $B_r$ such that $d<p_{k+1}-r$ and $d(p_{k+\ell}-r)<(d+1)(p_k-r)$. If we set $\displaystyle{n=dp_{k+\ell}\prod_{i=b(r)+1}^{k-1}p_i}$, then $n\in F_r$. 
\end{theorem} 
\begin{proof} 
First, note that $n\in B_r$ and 
\begin{equation} \label{Eq2.3} 
S_r(n)\leq d(p_{k+\ell}-r)\prod_{i=b(r)+1}^{k-1}(p_i-r). 
\end{equation} 
Using the hypothesis $d(p_{k+\ell}-r)<(d+1)(p_k-r)$, we get 
\begin{equation} \label{Eq2.4} 
S_r(n)<(d+1)\prod_{i=b(r)+1}^k(p_i-r), 
\end{equation} 
from which the hypothesis $d<p_{k+1}-r$ yields 
\begin{equation} \label{Eq2.5} 
S_r(n)<\prod_{i=b(r)+1}^{k+1}(p_i-r). 
\end{equation} 
Now, choose some arbitrary $m\in B_r$ with $m>n$. We will show that $S_r(m)>S_r(n)$. There is a unique integer $t>b(r)$ such that 
\[\prod_{i=b(r)+1}^tp_i\leq m<\prod_{i=b(r)+1}^{t+1}p_i.\]  
Clearly, $\omega(m)\leq t-b(r)$, so $\displaystyle{\frac{S_r(m)}{m}\geq\prod_{i=b(r)+1}^t\left(1-\frac{r}{p_i}\right)}$. 
This implies that 
$\displaystyle{S_r(m)\geq\prod_{i=b(r)+1}^t(p_i-r)}$. If $t\geq k+1$, then we may use \eqref{Eq2.5} to conclude that $S_r(n)<S_r(m)$. Therefore, let us assume that $t\leq k$. Then $\omega(m)\leq k-b(r)$. Suppose $\omega(m)\leq k-1-b(r)$ so that $\displaystyle{\frac{S_r(m)}{m}\geq\prod_{i=b(r)+1}^{k-1}\left(1-\frac{r}{p_i}\right)}$. From \eqref{Eq2.3}, we have $\displaystyle{\frac{S_r(n)}{n}<\prod_{i=b(r)+1}^{k-1}\left(1-\frac{r}{p_i}\right)}$, so $\displaystyle{\frac{S_r(n)}{n}<\frac{S_r(m)}{m}}$. Because $m>n$, we see that $S_r(n)<S_r(m)$.  
\par 
Now, assume $\omega(m)=k-b(r)$. Then we may write $\displaystyle{m=\mu\prod_{i=1}^{k-b(r)}q_i}$, where $\mu$ is a positive integer whose prime factors are all in the set $\{q_1,q_2,\ldots,q_{k-b(r)}\}$ and, for all $i,j\in\{1,2,\ldots,k-b(r)\}$ with $i<j$, $q_i$ is a prime and $p_{b(r)+i}\leq q_i<q_j$. This means that $\displaystyle{S_r(m)=\mu\prod_{i=1}^{k-b(r)}(q_i-r)}$. If $\mu\geq d+1$, then we may use \eqref{Eq2.4} to find that $\displaystyle{S_r(n)<\mu\prod_{i=b(r)+1}^k(p_i-r)=\mu\prod_{i=1}^{k-b(r)}(p_{b(r)+i}-r)\leq S_r(m)}$. Hence, we may assume that $\mu\leq d$. Because $m>n$, we have 
\begin{equation} \label{Eq2.6}
\prod_{i=1}^{k-b(r)}q_i>\frac{d}{\mu}p_{k+\ell}\prod_{i=b(r)+1}^{k-1}p_i.
\end{equation} 
For each $i\in\{1,2,\ldots,k-1-b(r)\}$, let $x_i=p_{b(r)+i}$, and let $y_i=q_i$. If we set $s=k-1-b(r)$, $\displaystyle{X=\frac{d}{\mu}p_{k+\ell}}$, and $Y=q_{k-b(r)}$, then we may use Lemma \ref{Lem2.1} and \eqref{Eq2.6} to conclude that     
\[\prod_{i=1}^{k-b(r)}(q_i-r)>\left(\frac{d}{\mu}p_{k+\ell}-r\right)\prod_{i=b(r)+1}^{k-1}(p_i-r).\] 
Thus, because $\mu\leq d$, we have 
\[S_r(m)=\mu\prod_{i=1}^{k-b(r)}(q_i-r)>(dp_{k+\ell}-r\mu)\prod_{i=b(r)+1}^{k-1}(p_i-r)\] 
\[\geq d(p_{k+\ell}-r)\prod_{i=b(r)+1}^{k-1}(p_i-r).\] 
Recalling \eqref{Eq2.3}, we have $S_r(m)>S_r(n)$, so the proof is complete. 
\end{proof} 
\section{Prime Divisors of Sparsely Schemmel \\ Totient Numbers}
In their paper, Masser and Shiu casually mention that $2$ is the only sparsely totient prime power \cite{Masser86}, but their brief proof utilizes the fact that, for $r=1$, $r+1$ is prime. We will see that if $r+1$ is prime, then $r+1$ is indeed the only sparsely Schemmel totient number of order $r$ that is a prime power. However, if $r+1$ is composite, there could easily be multiple sparsely Schemmel totient numbers  of order $r$ that are prime powers. The following results will provide an upper bound (in terms of $r$) for the values of sparsely Shemmel totient prime powers of order $r$. 
\begin{lemma} \label{Lem3.1} 
If $j\in\mathbb{N}\backslash\{1,2,4\}$, then $\displaystyle{\frac{p_{j+1}}{p_j}\leq\frac{7}{5}}$. 
\end{lemma}
\begin{proof}
Pierre Dusart \cite{Dusart10} has shown that, for $x\geq 396\hspace{0.75 mm} 738$, there must be at  least one prime in the interval $\displaystyle{\left[x, x+\frac{x}{25\log^2x}\right]}$. Therefore, whenever $p_j>396\hspace{0.75 mm} 738$, we may set $x=p_j+1$ to get $\displaystyle{p_{j+1}\leq (p_j+1)+\frac{p_j+1}{25\log^2(p_j+1)}}$ $\displaystyle{<\frac{7}{5}p_j}$. Using Mathematica 9.0 \cite{Wolfram09}, we may quickly search through all the primes less than $396\hspace{0.75 mm} 738$ to conclude the desired result.   
\end{proof} 
\begin{lemma} \label{Lem3.2}
Let $p$ be a prime, and let $r$, $\alpha$, and $\gamma$ be positive integers such that $\alpha>1$ and $p\nmid\gamma$. If $p^{\alpha}\gamma\in F_r$, then $p^{\alpha-1}\gamma\in F_r$. 
\end{lemma} 
\begin{proof} 
Suppose, for the sake of finding a contradiction, that $p^{\alpha-1}\gamma\not\in F_r$ and $p^{\alpha}\gamma\in F_r$. Because $p^{\alpha}\gamma\in F_r\subseteq B_r$, we know that $p^{\alpha-1}\gamma\in B_r$. Then, because $p^{\alpha-1}\gamma\not\in F_r$, there must exist some $m\in B_r$ such that $m>p^{\alpha-1}\gamma$ and $S_r(m)\leq S_r(p^{\alpha-1}\gamma)=p^{\alpha-2}(p-r)S_r(\gamma)$. However, this implies that $pm>p^{\alpha}\gamma$ and $S_r(pm)\leq pS_r(m)\leq p^{\alpha-1}(p-r)S_r(\gamma)=S_r(p^{\alpha}\gamma)$, which contradicts the fact that $p^{\alpha}\gamma\in F_r$. 
\end{proof} 
\begin{theorem} \label{Thm3.1} 
If $p$ is a prime and $r$ is a positive integer, then $p\in F_r$ if and only if $r<p<(p_{b(r)+1}-r)(p_{b(r)+2}-r)+r$.  
\end{theorem} 
\begin{proof} 
First, suppose $r<p<(p_{b(r)+1}-r)(p_{b(r)+2}-r)+r$, and let $m$ be an arbitrary element of $B_r$ that is greater than $p$. We will show that $S_r(m)>p-r$. If $\omega(m)\geq 2$, then $\displaystyle{S_r(m)\geq\prod_{i=1}^{\omega(m)}(p_{b(r)+i}-r)\geq(p_{b(r)+1}-r)(p_{b(r)+2}-r)}$ $>p-r$. Therefore, we may assume that $\omega(m)=1$ so that we may write $m=q^{\beta}$ for some prime $q>r$ and positive integer $\beta$. Furthermore, we may assume $\beta>1$ because if $\beta=1$, then $S_r(m)=q-r=m-r>p-r$. If $r\not\in\{1,2,3,5\}$, then it is easy to see, with the help of Lemma \ref{Lem3.1}, that $p_{b(r)+2}<2r$. Thus, if $r\not\in\{1,2,3,5\}$, then we have $S_r(m)=q^{\beta-1}(q-r)\geq q(q-r)\geq p_{b(r)+1}(p_{b(r)+1}-r)>r(p_{b(r)+1}-r)>(p_{b(r)+1}-r)(p_{b(r)+2}-r)>p-r$. If $r=1$, then the inequality $p<(p_{b(r)+1}-r)(p_{b(r)+2}-r)+r$ forces $p=2$, so $S_r(m)=q^{\beta-1}(q-r)>1=p-r$. If $r=2$, then the inequality $r<p<(p_{b(r)+1}-r)(p_{b(r)+2}-r)+r$ forces $p=3$, so $S_r(m)=q^{\beta-1}(q-r)>1=p-r$. If $r=3$, then $q\geq 5$ and either $p=5$ or $p=7$. Therefore, $S_r(m)=q^{\beta-1}(q-r)\geq5(5-3)>p-r$. Finally, if $r=5$, then $p\in\{7,11,13\}$ and $q\geq 7$. Thus, $S_r(m)=q^{\beta-1}(q-r)\geq7(7-5)>p-r$. 
\par 
To prove the converse, suppose $p\geq (p_{b(r)+1}-r)(p_{b(r)+2}-r)+r$. We wish to find some $m\in B_r$ such that $m>p$ and $S_r(m)\leq p-r$. We may assume that $p>p_{b(r)+1}p_{b(r)+2}$ because, otherwise, we may simply set $m=p_{b(r)+1}p_{b(r)+2}$. We know that there exists a unique integer $t\geq b(r)+2$ such that $p_{b(r)+1}p_t<p<p_{b(r)+1}p_{t+1}$. Suppose $r>3$ so that, with the help of Lemma \ref{Lem3.1} and some very short casework, we may conclude that $\displaystyle{p_{b(r)+1}\leq\frac{11}{7}r}$ and $\displaystyle{p_{t+1}\leq\frac{11}{7}p_t}$. Then, setting $m=p_{b(r)+1}p_{t+1}$, we have 
\[S_r(m)=(p_{b(r)+1}-r)(p_{t+1}-r)\leq\frac{4}{7}r\left(\frac{11}{7}p_t-r\right)\] 
\[<\frac{44}{49}p_{b(r)+1}p_t-\frac{4}{7}r^2<p_{b(r)+1}p_t-r<p-r.\]   
We now handle the cases in which $r\leq 3$. If $r=1$, then $p$ is odd, so we may set $m=2p$ to get $S_1(m)=S_1(2)S_1(p)=p-1=p-r$. If $r=2$, then $3\nmid p$, so we may set $m=3p$ to find $S_2(m)=S_2(3)S_2(p)=p-2=p-r$. Finally, if $r=3$, then we have $5p_t<p<5p_{t+1}$. Set $m=5p_{t+1}$. As $\displaystyle{p_{t+1}<\frac{5}{2}p_t}$, we have $2p_{t+1}-3<p$, so $S_3(m)=2(p_{t+1}-3)<p-3=p-r$.  
\end{proof} 
\begin{theorem} \label{Thm3.2} 
Let $r$ be a positive integer, and let $p$ be a prime. Then $p_{b(r)+1}^2\not\in F_r$ and $p^3\not\in F_r$. 
\end{theorem} 
\begin{proof} 
Suppose $p_{b(r)+1}^2\in F_r$. Then, as $p_{b(r)+1}p_{b(r)+2}>p_{b(r)+1}^2$, we must have $(p_{b(r)+1}-r)(p_{b(r)+2}-r)>p_{b(r)+1}(p_{b(r)+1}-r)$. Therefore, $r<p_{b(r)+2}-p_{b(r)+1}$. It is easy to see that this inequality fails to hold for all $r\leq 10$. For $r\geq 11$, we may use Lemma \ref{Lem3.1} to write $p_{b(r)+1}<\sqrt{2}r$ and $p_{b(r)+2}<\sqrt{2}p_{b(r)+1}$. Hence, $p_{b(r)+2}-p_{b(r)+1}<(\sqrt{2}-1)p_{b(r)+1}<(2-\sqrt{2})r<r$, which is a contradiction. 
\par 
Now, suppose $p^3\in F_r$. Then, by Lemma \ref{Lem3.2}, we know that $p^2\in F_r$, so $p>p_{b(r)+1}$. Let $t$ be the unique integer such that $p_{b(r)+1}p_t<p^2<p_{b(r)+1}p_{t+1}$. Then $p^3<p_{b(r)+1}p_{t+1}p$ and $p_{b(r)+1}<p<p_{t+1}$. Therefore, as $p^3\in F_r$, we see that $S_r(p^3)=p^2(p-r)<(p_{b(r)+1}-r)(p_{t+1}-r)(p-r)$, implying that $p^2<(p_{t+1}-r)(p_{b(r)+1}-r)<p_{t+1}(p_{b(r)+1}-r)$. Using Bertrand's Postulate, we see that $p_{t+1}<2p_t$ and $p_{b(r)+1}\leq 2r$. Therefore, $p_{b(r)+1}p_t<p^2<p_{t+1}(p_{b(r)+1}-r)<2p_t(p_{b(r)+1}-r)$, so $2r<p_{b(r)+1}$. This is our desired contradiction. 
\end{proof} 
Combining Lemma \ref{Lem3.2}, Theorem \ref{Thm3.1}, and Theorem \ref{Thm3.2}, we see that any $n\in F_r$ satisfying $n\geq((p_{b(r)+1}-r)(p_{b(r)+2}-r)+r)^2$ must have at least two prime factors. Furthermore, we record the following conjecture about the nonexistence of sparsely Schemmel totient numbers that are squares of primes. 
\begin{conjecture} \label{Conj3.1} 
For any prime $p$ and positive integer $r$, $p^2\not\in F_r$. 
\end{conjecture}  
We now proceed to establish asymptotic results concerning the primes that divide and do not divide sparsely Schemmel totient numbers. For a given $r\in\mathbb{N}$ and $n\in F_r$, we will define $P_k(n)$ to be the $k^{th}$ largest prime divisor of $n$ (provided $\omega(n)\geq k$), and we will let $Q_k(n)$ denote the $k^{th}$ smallest prime that is larger than $r$ and does not divide $n$ (the functions $Q_k$ depends on $r$, but this should not lead to confusion because we will work with fixed values of $r$).
We will let $\displaystyle{R(n)=n\prod_{\substack{p\in\mathbb{P} \\ p\vert n}}p^{-1}}$. We will also make use of the Jacobsthal function $J$. For a positive integer $n$, $J(n)$ is defined to be the smallest positive integer $a$ such that every set of $a$ consecutive integers contains an element that is relatively prime to $n$. In particular, for any positive integer $r$, $J(r\#)$ is the largest possible difference between consecutive elements of $B_r$. For convenience, we will write $J_r=J(r\#)$. Finally, we will let $\lambda_k(r)$ be the unique positive real root of the polynomial $\displaystyle{\frac{J_r}{r}x^k+kx-(k-1)}$.
\begin{lemma} \label{Lem3.3} 
If $r$, $n$, and $k$ are positive integers such that $k\geq 2$, $n\in F_r$, and $\omega(n)\geq k$, then $Q_{k-1}(n)>\lambda_k(r)(P_k(n)-r)$. 
\end{lemma} 
\begin{proof} 
Write $\displaystyle{M=\prod_{i=1}^kP_i(n)}$ and $\displaystyle{N=\prod_{i=1}^{k-1}Q_i(n)}$. Let $\mu$ be the smallest element of $B_r$ that is greater than $\displaystyle{\frac{M}{N}}$. Because any set of $J_r$ consecutive integers must contain at least one element that is not divisible by any prime less than or equal to $r$, we find that $\displaystyle{\mu<\frac{M}{N}+J_r}$. Let us put $\displaystyle{m=\frac{\mu N}{M}n}$ so that $m\in B_r$ and $\displaystyle{1<\frac{m}{n}<1+\frac{J_rN}{M}<1+J_r\frac{Q_{k-1}(n)^{k-1}}{P_k(n)^k}}$. Because $m$ is divisible by all of the prime divisors of $N$ and all the prime divisors of $n$ except possibly  those that divide $M$, we have
\[\frac{S_r(m)}{m}\leq\prod_{i=1}^{k-1}\left(1-\frac{r}{Q_i(n)}\right)\prod_{j=1}^k\left(1-\frac{r}{P_j(n)}\right)^{-1}\frac{S_r(n)}{n}\]
\[<\left(1-\frac{r}{Q_{k-1}(n)}\right)^{k-1}\left(1-\frac{r}{P_k(n)}\right)^{-k}\frac{S_r(n)}{n}.\]
This implies that 
\[S_r(m)<\left(1+J_r\frac{Q_{k-1}(n)^{k-1}}{P_k(n)^k}\right)\left(1-\frac{r}{Q_{k-1}(n)}\right)^{k-1}\left(1-\frac{r}{P_k(n)}\right)^{-k}S_r(n),\] 
so the fact that $n\in F_r$ implies that 
\begin{equation} \label{Eq3.1} 
\left(1+J_r\frac{Q_{k-1}(n)^{k-1}}{P_k(n)^k}\right)\left(1-\frac{r}{Q_{k-1}(n)}\right)^{k-1}\left(1-\frac{r}{P_k(n)}\right)^{-k}>1. 
\end{equation}  
Write $\displaystyle{x_1=J_r\frac{Q_{k-1}(n)^{k-1}}{P_k(n)^k}}$, $\displaystyle{x_2=\frac{r}{Q_{k-1}(n)}}$, and $\displaystyle{x_3=\frac{r}{P_k(n)}
}$ so that \eqref{Eq3.1} becomes $(1+x_1)(1-x_2)^{k-1}(1-x_3)^{-k}>1$. Because $x_1$ and $x_2$ are positive and $0<x_3<1$, we may invoke the inequalities $1+x_1<e^{x_1}$, $1-x_2<e^{-x_2}$, and $(1-x_3)^{-1}<e^{x_3/(1-x_3)}$ to write 
\begin{equation} \label{Eq3.2} 
e^{x_1-(k-1)x_2+kx_3/(1-x_3)}>1.
\end{equation} 
After a little algebraic manipulation, \eqref{Eq3.2} becomes 
\[\frac{J_r}{r}\left(\frac{Q_{k-1}(n)}{P_k(n)}\right)^k+k\frac{Q_{k-1}(n)}{P_k(n)-r}-(k-1)>0.\] 
Thus, if we write $\displaystyle{A(x)=\frac{J_r}{r}x^k+kx-(k-1)}$, then $\displaystyle{A\left(\frac{Q_{k-1}(n)}{P_k(n)-r}\right)>0}$. This means that $\displaystyle{\frac{Q_{k-1}(n)}{P_k(n)-r}>\lambda_k(r)}$, so we are done.     
\end{proof} 
\begin{lemma} \label{Lem3.4} 
For any positive integers $r$ and $n$ with $\omega(n)\geq 2$ and $n\in F_r$, \[P_1(n)<Q_1(n)\left(1-J_r+\frac{J_r}{r}Q_1(n)\right).\]   
\end{lemma}  
\begin{proof} 
Fix $r$ and $n$, and write $P=P_1(n)$ and $Q=Q_1(n)$. Suppose, for the sake of finding a contradiction, that \[P\geq Q\left(1-J_r+\frac{J_r}{r}Q\right).\] 
Let $\mu$ be the smallest element of $B_r$ that is greater than $\displaystyle{\frac{P}{Q}}$. Then $\displaystyle{\mu<\frac{P}{Q}+J_r}$. Write $\displaystyle{m=\frac{Q\mu}{P}n}$ so that $m\in B_r$ and $\displaystyle{1<\frac{m}{n}<1+\frac{J_rQ}{P}\leq 1+\frac{J_r}{1-J_r+\frac{J_r}{r}Q}}$. Because $m$ is divisible by $Q$ and all the prime divisors of $n$ except possibly $P$, we have 
\[\frac{S_r(m)}{m}\leq\left(1-\frac{r}{Q}\right)\left(1-\frac{r}{P}\right)^{-1}\frac{S_r(n)}{n}.\] 
Therefore, 
\[S_r(m)<\left(1-\frac{r}{Q}\right)\left(1-\frac{r}{P}\right)^{-1}\left(1+\frac{J_r}{1-J_r+\frac{J_r}{r}Q}\right)S_r(n)\] 
\[\leq\left(1-\frac{r}{Q}\right)\left(1-\frac{r}{Q\left(1-J_r+\frac{J_r}{r}Q\right)}\right)^{-1}\left(1+\frac{J_r}{1-J_r+\frac{J_r}{r}Q}\right)S_r(n)\]
\[=\left(1-\frac{r}{Q}\right)\left(1-J_r+\frac{J_r}{r}Q-\frac{r}{Q}\right)^{-1}\left(1+\frac{J_r}{r}Q\right)S_r(n)=S_r(n).\]
This is our desired contradiction, so the proof is complete. 
\end{proof} 
\begin{lemma} \label{Lem3.5} 
Let $r$ be a positive integer, and let $n\in F_r$. Then 
\[R(n)<\frac{J_r}{r}Q_1(n)(Q_1(n)-r).\]   
\end{lemma} 
\begin{proof} 
Fix $r$, and $n$, and write $Q=Q_1(n)$ and $R=R(n)$. Suppose \\ 
$\displaystyle{R\geq\frac{J_r}{r}Q(Q-r)}$. Let $\mu$ be the smallest element of $B_r$ greater than $\displaystyle{\frac{R}{Q}}$. Then $\displaystyle{\mu<\frac{R}{Q}+J_r}$. If we put $\displaystyle{m=\frac{Q\mu}{R}n}$, then $m\in B_r$ and 
\[1<\frac{m}{n}<1+\frac{J_rQ}{R}\leq 1+\frac{r}{Q-r}.\] 
Because $m$ is divisible by $Q$ and all the prime divisors of $n$, we have \\ 
$\displaystyle{\frac{S_r(m)}{m}\leq\left(1-\frac{r}{Q}\right)\frac{S_r(n)}{n}}$. This implies that 
\[S_r(m)<\left(1-\frac{r}{Q}\right)\left(1+\frac{r}{Q-r}\right)S_r(n)=S_r(n),\] 
which is a contradiction.   
\end{proof} 
\begin{corollary} \label{Cor3.1} 
Let $r$ be a positive integer. 
Then, for $n\in F_r$,
\[\lim_{n\rightarrow\infty}\omega(n)=\infty.\]  
\end{corollary} 
\begin{proof} 
Suppose otherwise. Then there exists some positive integer $\Omega$ such that there are arbitrarily large values of $n\in F_r$ satisfying $\omega(n)<\Omega$. This implies that there are arbitrarily large values of $n\in F_r$ satisfying $Q_1(n)\leq p_{b(r)+\Omega}$. By Lemma \ref{Lem3.4}, this implies that there exists some integer $N$ such that there are arbitrarily large values of $n\in F_r$ satisfying $P_1(n)\leq N$ and $Q_1(n)\leq N$. However, if $P_1(n)\leq N$, then $\displaystyle{R(n)\geq \frac{n}{N\#}}$. Using Lemma \ref{Lem3.5}, we see that $\displaystyle{\frac{n}{N\#}\leq R(n)<\frac{J_r}{r}Q_1(n)(Q_1(n)-r)\leq\frac{J_r}{r}N(N-r)}$, which is a contradiction because $n$ can be arbitrarily large.  
\end{proof}
\begin{corollary} \label{Cor3.2} 
Let $r$ be a positive integer. For sufficiently large $n\in F_r$, $P_1(n)^4\nmid n$.
\end{corollary} 
\begin{proof} 
For any integer $n>1$, write $\upsilon_{P_1(n)}(n)=\eta(n)$. Using Lemma \ref{Lem3.5}, we see that, for any $n\in F_r$ satisfying $n>1$,  
\[P_1(n)^{\eta(n)-1}\leq R(n)<\frac{J_r}{r}Q_1(n)(Q_1(n)-r).\] Because $Q_1(n)$ is at most the smallest prime exceeding $P_1(n)$, we may use Bertrand's Postulate to write 
\[\frac{J_r}{r}Q_1(n)(Q_1(n)-r)\leq2\frac{J_r}{r}P_1(n)(2P_1(n)-r)<4\frac{J_r}{r}P_1(n)^2.\] 
If $P_1(n)^4\vert n$, then $\eta(n)-1\geq 3$, so $\displaystyle{P_1(n)<4\frac{J_r}{r}}$. By Corollary \ref{Cor3.1}, we see that this is impossible for sufficiently large $n$. 
\end{proof} 
Masser and Shiu show that $P_1(n)^3\nmid n$ for all sparsely totient numbers $n$, but their methods are not obviously generalizable \cite{Masser86}. Thus, we make the following conjecture. 
\begin{conjecture} \label{Conj3.2} 
For any positive integers $r$ and $n$ with $n\in F_r$ and $n>1$, $P_1(n)^3\nmid n$. 
\end{conjecture}  
For small values of $r$, we may effortlessly make small amounts of progress toward Conjecture \ref{Conj3.2}. For example, it is easy to use Lemma \ref{Lem3.5} to show that $P_1(n)^4\nmid n$ for all $n\in F_2$.  Indeed, if $P_1(n)^4\vert n$ for some $n\in F_2$, then $\displaystyle{P_1(n)<4\cdot\frac{J_2}{2}=4}$. This forces $n$ to be a power of $3$, but Theorem \ref{Thm3.2} tells us that there are no powers of $3$ in $F_2$ except $3$ itself.  
\par 
We are finally ready to establish our promised asymptotic results. 
\begin{theorem} \label{Thm3.3} 
Let $r$, $K$, and $L$ be positive integers with $K\geq 2$. For $n\in F_r$, we have 
\begin{enumerate}[(a)] 
\item $\displaystyle{\limsup_{n\rightarrow\infty}\frac{P_1(n)}{\log n}\geq 2}$,
\item $\displaystyle{\limsup_{n\rightarrow\infty}\frac{Q_L(n)}{\log n}=1}$,
\item $\displaystyle{\limsup_{n\rightarrow\infty}\frac{P_K(n)}{\log n}\leq \lambda_K(r)^{-1}}$,  
\item $\displaystyle{\limsup_{n\rightarrow\infty}\frac{P_1(n)}{\log^2 n}\leq\frac{J_r}{r}}$.  
\end{enumerate}  
\end{theorem} 
\begin{proof} 
To prove $(a)$, let us begin by choosing some integer $k\geq b(r)+2$. Let $\ell(k)$ be the largest integer such that $p_{k+\ell(k)}<2p_k-r$. Setting
$\displaystyle{n(k)=p_{k+\ell(k)}\prod_{i=b(r)+1}^{k-1}p_i}$,  
we see, by Theorem \ref{Thm2.1}, that $n(k)\in F_r$. Furthermore, as $\displaystyle{\prod_{i=b(r)+1}^kp_i\leq n(k)<\prod_{i=b(r)+1}^{k+1}p_i}$, the Prime Number Theorem tells us that $p_k\sim\log n(k)$ as $k\rightarrow\infty$. Thus, as $k\rightarrow\infty$, $P_1(n(k))=p_{k+\ell(k)}\sim 2p_k\sim2\log n(k)$. 
\par 
To prove $(b)$, choose any $n\in F_r$ with $n>1$, and let $k(n)$ be the unique integer satisfying $\displaystyle{\prod_{i=b(r)+1}^{k(n)}p_i\leq n<\prod_{i=b(r)+1}^{k(n)+1}p_i}$. Using the Prime Number Theorem again, we have $Q_L(n)\leq p_{k(n)+L}\sim\log n$ as $n\rightarrow\infty$. In addition, for those $n\in F_r$ (guaranteed by Theorem \ref{Thm2.1}) of the form $\displaystyle{n=\prod_{i=b(r)+1}^{k(n)}p_i}$, we see that $Q_L(n)=p_{k(n)+L}\sim\log n$.   
\par 
Corollary \ref{Cor3.1} guarantees that the limit in $(c)$ is well-defined. To prove the limit, we use Lemma \ref{Lem3.3} to find that if $n\in F_r$ and $\omega(n)\geq K$, then 
\[\frac{P_K(n)}{\log n}<\lambda_K(r)^{-1}\frac{Q_{K-1}(n)}{\log n}+\frac{r}{\log n}.\]
Then the desired result follows from setting $L=K-1$ in $(b)$. 
\par 
Finally, $(d)$ follows immediately from Lemma \ref{Lem3.4} and from setting $L=1$ in $(b)$.  
\end{proof}

\end{document}